\documentclass[10pt]{article}
\usepackage{latexsym}
\usepackage{amssymb}
\usepackage{amsmath}
\usepackage{amsfonts}

\newtheorem{theorem}{Theorem}[section]

\newtheorem{proposition}[theorem]{Proposition}
\newtheorem{example}[theorem]{Example}

\newcommand{\eproof}{\makebox[1cm]{}\hfill{\framebox[2.5mm]{}}}

\usepackage[top=3cm, left=4cm, right=4cm, bottom=3cm, ignoreall]{geometry}

\newcommand{\mtr}[1]{(#1)}
\newcommand{\dtr}[1]{\langle #1\rangle}
\newcommand{\str}[1]{\{#1\}}

\begin{document}
\title{Flexible Latin directed triple systems}
\author
{Ale\v s Dr\'apal\thanks{Ale\v s Dr\'apal supported by VF 20102015006.},
Andrew R. Kozlik\thanks{Andrew Kozlik supported by SVV-2014-260107.}\\
Department of Algebra\\
Charles University\\
Sokolovsk\'a 83\\
186 75 Praha 8, CZECH REPUBLIC\\[2mm]
Terry S. Griggs\\
Department of Mathematics and Statistics\\
The Open University\\
Walton Hall\\
Milton Keynes MK7 6AA, UNITED KINGDOM}
\date{}
\maketitle

\vspace{-7mm}
\begin{abstract}
It is well known that, given a Steiner triple system, a quasigroup
can be formed by defining an operation $\cdot$ by the identities $
x \cdot x = x$ and $x \cdot y = z$ where $z$ is the third point in
the block containing the pair $\{x,y\}$. The same is true for a
Mendelsohn triple system where the pair $(x,y)$ is considered to
be ordered. But it is not true in general for directed triple
systems. However directed triple systems which form quasigroups
under this operation do exist and we call these Latin directed
triple systems. The quasigroups associated with Steiner and
Mendelsohn triple systems satisfy the flexible law $x \cdot (y
\cdot x) = (x \cdot y) \cdot x$ but those associated with Latin
directed triple systems need not. In a previous paper, [Discrete
Mathematics 312 (2012), 597--607], we studied non-flexible Latin
directed triple systems. In this paper we turn our attention to
flexible Latin directed triple systems.
\end{abstract}


\newpage
\section{Introduction}
This paper is a sequel to~\cite{ldts}. There, we introduced the
concepts of a Latin directed triple system and a DTS-quasigroup,
developed some of the basic theory and determined the existence
spectrum of Latin directed triple systems whose associated
quasigroups do not satisfy the flexible law. Here we turn our
attention to flexible Latin directed triple systems.

First we recall the basic definitions and results which are
appropriate for our purposes. A \emph{Steiner triple system} of
order~$n$, STS($n$), is a pair $(V,\mathcal{B})$ where $V$ is a
set of $n$ points and $\mathcal{B}$ is a collection of triples of
distinct points, also called blocks, taken from $V$ such that
every pair of distinct points from $V$ appears in precisely one
block. Such systems exist if and only if $n\equiv 1\text{ or }3
\pmod{6}$~\cite{K}. A \emph{Steiner quasigroup} or \emph{squag} or
\emph{idempotent totally symmetric quasigroup} is a pair
$(Q,\cdot)$ where $Q$ is a set and $\cdot$ is an operation on $Q$
satisfying the identities
\[
  x \cdot x = x, \quad y \cdot (x \cdot y) = x, \quad x \cdot y = y \cdot x.
\]
If $(V,\mathcal{B})$ is an STS($n$), then a Steiner quasigroup
$(Q,\cdot)$ is obtained by letting $Q = V$ and defining $x \cdot y
= z$ where $\{x,y,z\} \in \mathcal{B}$. The process is reversible;
if $Q$ is a Steiner quasigroup, then a Steiner triple system is
obtained by letting $V = Q$ and $\{x,y,z\} \in \mathcal{B}$ where
$x \cdot y = z$ for all $x, y \in Q$, $x \neq y$. Thus there is a
one-one correspondence between all Steiner triple systems and all
Steiner quasigroups~\cite[Theorem V.1.11]{P}. This is all
well-known.

Next consider ordered triples. There are two possibilities. A \emph{cyclically ordered triple},
denoted by $\mtr{x,y,z}$, contains the ordered pairs $(x,y)$, $(y,z)$, $(z,x)$ and
a \emph{transitively ordered triple}, denoted by $\dtr{x,y,z}$ contains the ordered
pairs $(x,y)$, $(y,z)$, $(x,z)$.

A \emph{Mendelsohn triple system} of
order~$n$, MTS($n$), is a pair $(V,\mathcal{B})$ where $V$ is a set of $n$ points and
$\mathcal{B}$ is a collection of cyclically ordered triples of distinct points taken
from $V$ such that every ordered pair of distinct points from $V$ appears in
precisely one triple. Such systems exist if and only if $n\equiv 0$ or 1 (mod 3), $n
\neq 6$~\cite{Me}. Quasigroups can be obtained from Mendelsohn triple
systems by precisely the same procedures as described above for Steiner triple
systems. Note that the law $y \cdot (x \cdot y) = x$ is usually called
semi-symmetric. So the quasigroups are known as \emph{idempotent semisymmetric
quasigroups}~\cite[Remark 2.12]{BL} or \emph{Mendelsohn quasigroups};
they satisfy the same properties as their Steiner counterparts with the exception of
commutativity. Similarly there is a one-one correspondence between Mendelsohn triple
systems and Mendelsohn quasigroups.

A \emph{directed triple system} of order $n$, DTS($n$), is a pair
$(V,\mathcal{B})$ where $V$ is a set of $n$ points and
$\mathcal{B}$ is a collection of transitively ordered triples of
distinct points taken from $V$ such that every ordered pair of
distinct points from $V$ appears in precisely one triple. Such
systems exist if and only if $n\equiv 0\text{ or }1
\pmod{3}$~\cite{HM}. Given a DTS($n$), an algebraic structure
$(V,\cdot)$ can be obtained as above by defining $x \cdot x = x$
and $x \cdot y = z$ for all $x, y \in V$, $x \neq y$ where $z$ is
the third element in the transitive triple containing the ordered
pair $(x,y)$. However the structure obtained is not necessarily a
quasigroup. If $\langle u,x,y \rangle$ and $\langle y,v,x \rangle
\in \mathcal{B}$ then $u \cdot x = v \cdot x = y$. But some
DTS($n$)s do yield quasigroups. Such a DTS($n$) will be called a
\emph{Latin directed triple system}, and denoted by LDTS($n$), to
reflect the fact that in this case the operation table forms a
Latin square. We call the quasigroup so obtained a
\emph{DTS-quasigroup}. In~\cite{ldts} an easy necessary and
sufficient condition for a directed triple system to be Latin was
proved.
\begin{theorem}\label{LDTScondx}
Let $D=(V,\mathcal{B})$ be a DTS($n$). Then $D$ is an LDTS($n$) if and only if
$\dtr{x,y,z} \in \mathcal{B} \Rightarrow \dtr{w,y,x} \in \mathcal{B}$ for some $w \in V$.
\end{theorem}

Before proceeding further, it is important to point out two
fundamental differences between DTS-quasigroups and Steiner or
Mendelsohn quasigroups which motivates the study of these
structures. First, DTS-quasigroups are \emph{not} in one-one
correspondence with Latin directed triple systems. Non-isomorphic
LDTS($n$)s can yield identical DTS-quasigroups. In view of this,
for purposes of enumeration, it makes more sense to count
non-isomorphic DTS-quasigroups rather than non-isomorphic
LDTS($n$)s. Secondly all Steiner and Mendelsohn quasigroups
satisfy the flexible law $x\cdot(y\cdot x) = (x\cdot y)\cdot x$.
DTS-quasigroups need not. In~\cite{ldts}, the two following
results were proved.
\begin{theorem}\label{enum}
The number of non-isomorphic DTS-quasigroups of order $n = 3$, $4$, $6$, $7$, $9$, $10$, $12$
are $1$, $0$, $0$, $2$, $4$, $0$, $2$ respectively.
\end{theorem}
\begin{theorem}\label{nonflexth}
The existence spectrum of non-flexible LDTS($n$)s is $n \equiv 0, 1 \pmod{3}$,
$n \neq 3$, $4$, $6$, $7$, $10$.
\end{theorem}

For flexible DTS-quasigroups, again there is an easy necessary and sufficient condition.
\begin{theorem}\label{flexcondx}
A DTS-quasigroup obtained from an LDTS($n$), $D=(V,\mathcal{B})$, satisfies the flexible
law if and only if $\dtr{x,y,z}\in\mathcal{B} \Rightarrow \dtr{x,z\cdot x,y\cdot x}\in\mathcal{B}$.
\end{theorem}

Note that trivially a Steiner quasigroup is a DTS-quasigroup. Such
a DTS-quasigroup will be called \emph{improper}; all others are
\emph{proper}. From~\cite{ldts}, there exist only two
non-isomorphic proper flexible DTS-quasigroups of order less
than~$13$; one of order~$7$ and one of order~$9$. They are given
in the two examples below, and in the same format as
in~\cite{ldts}, as Latin directed triple systems. For simplicity
commas are omitted from the triples. The set $\mathcal{T}$ is the
set of unordered triples or \emph{Steiner triples}. Each triple
$\str{x,y,z}$ represents a pair of transitively ordered triples in
one of three ways, (i)~$\dtr{x,y,z}$ and $\dtr{z,y,x}$; or
(ii)~$\dtr{y,z,x}$ and $\dtr{x,z,y}$; or (iii)~$\dtr{z,x,y}$ and
$\dtr{y,x,z}$. Thus these triples are \emph{bidirectional}. The
set $\mathcal{D}$ is a set of transitively ordered triples or
\emph{unidirectional} triples. Replacing a pair of bidirectional
triples in an LDTS($n$), say $\dtr{x,y,z}$ and $\dtr{z,y,x}$, with
a different pair of bidirectional triples, say $\dtr{y,z,x}$ and
$\dtr{x,z,y}$, gives a system which yields the same DTS-quasigroup
as the first and yet the two LDTS($n$)s may be non-isomorphic,
~\cite[Example 2.4]{ldts}.

\begin{example}\label{flex7}
Flexible LDTS($7$).\\
$V=\{0,1,2,3,4,5,6\}$.\\
$\mathcal{T}=\big\{\str{012}$, $\str{034}$, $\str{056}\big\}$ and\\
$\mathcal{D}=\big\{\dtr{315}$, $\dtr{514}$, $\dtr{416}$, $\dtr{613}$, $\dtr{326}$,
$\dtr{624}$, $\dtr{425}$, $\dtr{523}\big\}$.
\end{example}

\begin{example}\label{flex9}
Flexible LDTS($9$).\\
$V=\{0,1,2,3,4,5,6,7,8\}$.\\
$\mathcal{T}=\big\{\str{018}$, $\str{258}$, $\str{368}$, $\str{478}$, $\str{246}$, $\str{357}\big\}$ and\\
$\mathcal{D}=\big\{\dtr{207}$, $\dtr{706}$, $\dtr{605}$, $\dtr{504}$, $\dtr{403}$, $\dtr{302}$,
$\dtr{213}$, $\dtr{314}$, $\dtr{415}$, $\dtr{516}$, $\dtr{617}$, $\dtr{712}\big\}$.
\end{example}

\section{Structure of flexible DTS-quasigroups}
In a further paper~\cite{basics}, flexible DTS-quasigroups were
shown to have a certain structure in terms of their topology. Let
$D=(V,\mathcal{B})$ be an LDTS($n$). Denote by~$F$, the set of all
unordered triples $\{x,y,z\}$, where $\dtr{x,y,z}$ runs through
all unidirectional triples of~$D$. Now consider~$F$ as a set of
faces. Each edge $\{x,y\}$ is incident to two faces and hence we
get a generalized pseudosurface. By separating pinch points we
obtain a set of one or more components which are an invariant of
the LDTS($n$) and are very useful in determining whether two
DTS-quasigroups are isomorphic.

Consider a unidirectional triple $\dtr{z_1,x,z_0}\in\mathcal{B}$.
Then, using Theorem~\ref{LDTScondx}, there exists $k\geq 3$ and
points $z_0$, $z_1$, $z_2$, $\dots$, $z_{k-1}$ such that
\[
  \dtr{z_1,x,z_0}, \dtr{z_2,x,z_1}, \dots, \dtr{z_{k-1},x,z_{k-2}}, \dtr{z_0,x,z_{k-1}}\in\mathcal{B}.
\]
If $D$ is also flexible, using Theorem~\ref{flexcondx},
\[
  \dtr{z_1,y,z_2}, \dtr{z_2,y,z_3}, \dots, \dtr{z_{k-1},y,z_0}, \dtr{z_0,y,z_1}\in\mathcal{B}
\]
where $y = z_0\cdot z_1 = z_1\cdot z_2 = \cdots = z_{k-2}\cdot z_{k-1} = z_{k-1} \cdot z_0$.
These $2k$ transitive triples define a $k$-gonal bipyramid; denoted by~$O_k$, i.e.\ a graph
of $k+2$ vertices with a cycle of length~$k$, the points of which can be thought of as
situated around the equator of a sphere, and two middle vertex points which are connected
to all points of the cycle and which can be thought of as situated at the poles of the
sphere. Thus we have the following important result.
\begin{theorem}\label{33}
A flexible DTS-quasigroup of order~$n$ exists if and only if the
complete graph $K_n$ can be decomposed into triangles and graphs
$O_k$, $k\geq 3$. The components of the generalized pseudosurface
of the quasigroup are all spheres.
\end{theorem}

It is worth remarking that when at least one of the graphs $O_k$
has $k \geq 6$ and even, the decomposition of $K_n$ as described
in the theorem may also be used to obtain a non-flexible system.
Replace triples
\[\dtr{z_{2i+1},x,z_{2i}}, \dtr{z_{2i+2},x,z_{2i+1}}, \dtr{z_{2i},y,z_{2i+1}},
\dtr{z_{2i+1},y,z_{2i+2}}\] in the flexible system by triples
\[\dtr{z_{2i+1},z_{2i},x}, \dtr{x,z_{2i+2},z_{2i+1}}, \dtr{y,z_{2i},z_{2i+1}},
\dtr{z_{2i+1},z_{2i+2},y},\] $i=0,1,\ldots,(k-2)/2$, subscript
arithmetic modulo $k$. As illustration from Example~\ref{flex9},
the following is a non-flexible LDTS($9$).

\begin{example}\label{nonflex9}
Non-flexible LDTS($9$).\\
$V=\{0,1,2,3,4,5,6,7,8\}$.\\
$\mathcal{T}=\big\{\str{018}$, $\str{258}$, $\str{368}$, $\str{478}$, $\str{246}$, $\str{357}\big\}$ and\\
$\mathcal{D}=\big\{\dtr{270}$, $\dtr{076}$, $\dtr{650}$,
$\dtr{054}$, $\dtr{430}$, $\dtr{032}$, $\dtr{231}$, $\dtr{134}$,
$\dtr{451}$, $\dtr{156}$, $\dtr{671}$, $\dtr{172}\big\}$.
\end{example}

The number of equator cycles of each length is clearly an
invariant of a flexible DTS-quasigroup. Another invariant can be
calculated as follows. The \emph{type} of a vertex is the list of
valencies which it has as a middle vertex or pole of a bipyramid.
The number of vertices of each type is then also an invariant.

Thus, for the DTS-quasigroup given in Example~\ref{flex7}, there is an equator cycle of
length~$4$ $(3,5,4,6)$ and two points ($1$ and $2$) of type~$4$. For the DTS-quasigroup
given in Example~\ref{flex9} there is an equator cycle of length~$6$ $(2,7,6,5,4,3)$ and
two points ($0$ and $1$) of type~$6$. A more instructive example however of order~$13$ is
given below.
\begin{example}\label{flex13}
Flexible LDTS($13$).\\
$V=\{0,1,2,3,4,5,6,7,8,9,T,E,W\}$.\\
$\mathcal{T}$=$\big\{\str{045}$, $\str{179}$, $\str{1TW}$, $\str{358}$, $\str{37W}$, $\str{59T}\big\}$ and\\
$\mathcal{D}=\big\{\dtr{103}$, $\dtr{302}$, $\dtr{201}$, $\dtr{142}$, $\dtr{243}$,
$\dtr{341}$, $\dtr{629}$, $\dtr{92E}$, $\dtr{E2T}$, $\dtr{T26}$, $\dtr{63T}$, $\dtr{T3E}$,
$\dtr{E39}$, $\dtr{936}$, $\dtr{156}$, $\dtr{65W}$, $\dtr{W52}$, $\dtr{257}$, $\dtr{75E}$,
$\dtr{E51}$, $\dtr{18E}$, $\dtr{E87}$, $\dtr{782}$, $\dtr{28W}$, $\dtr{W86}$, $\dtr{681}$,
$\dtr{60E}$, $\dtr{E0W}$, $\dtr{W09}$, $\dtr{908}$, $\dtr{80T}$, $\dtr{T07}$, $\dtr{706}$,
$\dtr{647}$, $\dtr{74T}$, $\dtr{T48}$, $\dtr{849}$, $\dtr{94W}$, $\dtr{W4E}$, $\dtr{E46}\big\}$.
\end{example}
For this system there are equator cycles of length $3$, $4$, $6$ and $7$, two points
($2$ and $3$) of type~$4$, two points ($5$ and $8$) of type~$6$ and two points ($0$ and $4$)
of type $3,7$.

At $n=13$, the combinatorial explosion takes over and, as reported
in~\cite{basics}, there are $1\,206\,969$ non-isomorphic
DTS-quasigroups of order~$13$. Details of their automorphism
groups and genera of their separated surface components are also
given in that paper. However, only $924$ of these quasigroups are
flexible including the two Steiner quasigroups of this order.
Table~\ref{tbl:class} shows the classification in terms of numbers
of Steiner triples~($t$), lengths of equator cycles, and types of
middle valency points.

\begin{table}
\begin{center}
\footnotesize
\begin{tabular}{c@{\ }|@{\ }c@{\ }|@{\ }c@{\ }|@{\ }c@{\ }|@{\ }c@{\ }|@{\ }c@{\ }|@{\ }c@{\ }|@{\ }c@{\ }|@{\ }c@{\ }|@{\ }c@{\ }|@{\ }c@{\ }|@{\ }c@{\ }|@{\ }c@{\ }|@{\ }c@{\ }|@{\ }c@{\ }|@{\ }c@{\ }|@{\ }c@{\ }|@{\ }c@{\ }|@{\ }c@{\ }|@{\ }c@{\ }|@{\ }c}
& \multicolumn{7}{c|@{\ }}{\bf Equator Cycles} & & \multicolumn{12}{c}{\bf Middle Valency Points}\\
 \bf\# & $\bf3$ & $\bf4$ & $\bf5$ & $\bf6$ & $\bf7$ & $\bf8$ & $\bf10$ & $\boldsymbol{t}$ & $\bf4$ & $\bf6$ & $\bf3,3$ & $\bf8$ & $\bf5,3$ & $\bf4,4$ & $\bf10$ & $\bf7,3$ & $\bf6,4$ & $\bf5,5$ & $\bf4,3,3$ & $\bf4,4,4$\\
\hline
76 &    &  2 &    &  1 &    &    &    & 12 &  4 &  2 &    &    &    &    &    &    &    &    &    & \\[-1pt]
\hline
69 &    &  1 &    &    &    &    &  1 & 12 &  2 &    &    &    &    &    &  2 &    &    &    &    & \\[-1pt]
\hline
56 &    &  2 &    &    &    &    &  1 &  8 &  4 &    &    &    &    &    &  2 &    &    &    &    & \\[-1pt]
\hline
46 &    &  1 &    &  1 &    &    &    & 16 &  2 &  2 &    &    &    &    &    &    &    &    &    & \\[-1pt]
\hline
41 &    &  1 &    &  2 &    &    &    & 10 &  2 &  4 &    &    &    &    &    &    &    &    &    & \\[-1pt]
\hline
36 &    &  3 &    &    &    &    &    & 14 &  6 &    &    &    &    &    &    &    &    &    &    & \\[-1pt]
\hline
36 &    &  3 &    &    &    &    &    & 14 &  4 &    &    &    &    &  1 &    &    &    &    &    & \\[-1pt]
\hline
36 &    &  2 &    &  1 &    &    &    & 12 &  2 &    &    &    &    &    &    &    &  2 &    &    & \\[-1pt]
\hline
32 &  1 &  2 &    &    &  1 &    &    &  8 &  4 &    &    &    &    &    &    &  2 &    &    &    & \\[-1pt]
\hline
32 &    &  2 &    &  1 &    &    &    & 12 &  2 &  2 &    &    &    &  1 &    &    &    &    &    & \\[-1pt]
\hline
28 &    &  3 &    &  1 &    &    &    &  8 &  6 &  2 &    &    &    &    &    &    &    &    &    & \\[-1pt]
\hline
28 &    &  3 &    &  1 &    &    &    &  8 &  4 &  2 &    &    &    &  1 &    &    &    &    &    & \\[-1pt]
\hline
28 &    &  3 &    &  1 &    &    &    &  8 &  4 &    &    &    &    &    &    &    &  2 &    &    & \\[-1pt]
\hline
28 &    &  2 &    &    &    &    &    & 18 &  4 &    &    &    &    &    &    &    &    &    &    & \\[-1pt]
\hline
28 &    &    &    &  1 &    &    &  1 & 10 &    &  2 &    &    &    &    &  2 &    &    &    &    & \\[-1pt]
\hline
24 &    &  2 &    &  2 &    &    &    &  6 &  2 &  2 &    &    &    &    &    &    &  2 &    &    & \\[-1pt]
\hline
24 &  1 &  1 &    &    &  1 &    &    & 12 &  2 &    &    &    &    &    &    &  2 &    &    &    & \\[-1pt]
\hline
22 &    &  1 &    &  2 &    &    &    & 10 &    &  2 &    &    &    &    &    &    &  2 &    &    & \\[-1pt]
\hline
18 &    &    &    &  2 &    &    &    & 14 &    &  4 &    &    &    &    &    &    &    &    &    & \\[-1pt]
\hline
16 &    &  2 &    &    &    &    &  1 &  8 &  2 &    &    &    &    &  1 &  2 &    &    &    &    & \\[-1pt]
\hline
13 &    &    &    &    &    &    &  1 & 16 &    &    &    &    &    &    &  2 &    &    &    &    & \\[-1pt]
\hline
12 &  2 &  1 &    &    &    &  1 &    &  8 &  2 &    &  2 &  2 &    &    &    &    &    &    &    & \\[-1pt]
\hline
12 &  1 &  1 &    &  1 &  1 &    &    &  6 &  2 &  2 &    &    &    &    &    &  2 &    &    &    & \\[-1pt]
\hline
12 &    &  3 &    &  1 &    &    &    &  8 &  2 &  2 &    &    &    &  2 &    &    &    &    &    & \\[-1pt]
\hline
12 &  1 &    &    &  1 &  1 &    &    & 10 &    &  2 &    &    &    &    &    &  2 &    &    &    & \\[-1pt]
\hline
12 &    &  1 &    &  1 &    &    &  1 &  6 &  2 &  2 &    &    &    &    &  2 &    &    &    &    & \\[-1pt]
\hline
10 &  1 &    &    &    &  1 &  1 &    &  8 &    &    &    &  2 &    &    &    &  2 &    &    &    & \\[-1pt]
\hline
 9 &    &    &    &  3 &    &    &    &  8 &    &  6 &    &    &    &    &    &    &    &    &    & \\[-1pt]
\hline
 8 &    &  3 &    &  1 &    &    &    &  8 &  2 &    &    &    &    &  1 &    &    &  2 &    &    & \\[-1pt]
\hline
 8 &    &  2 &    &  2 &    &    &    &  6 &  2 &  4 &    &    &    &  1 &    &    &    &    &    & \\[-1pt]
\hline
 8 &  2 &    &    &    &    &    &  1 & 10 &    &    &  2 &    &    &    &  2 &    &    &    &    & \\[-1pt]
\hline
 8 &    &  2 &    &    &    &    &    & 18 &  2 &    &    &    &    &  1 &    &    &    &    &    & \\[-1pt]
\hline
 6 &  3 &    &    &    &  1 &    &    & 10 &    &    &  2 &    &    &    &    &  2 &    &    &    & \\[-1pt]
\hline
 6 &  2 &  1 &    &    &    &  1 &    &  8 &    &    &    &  2 &    &    &    &    &    &    &  2 & \\[-1pt]
\hline
 6 &    &  1 &    &  3 &    &    &    &  4 &    &  4 &    &    &    &    &    &    &  2 &    &    & \\[-1pt]
\hline
 6 &  2 &    &    &    &    &  1 &    & 12 &    &    &  2 &  2 &    &    &    &    &    &    &    & \\[-1pt]
\hline
 6 &    &  1 &  2 &    &    &    &    & 12 &  2 &    &    &    &    &    &    &    &    &  2 &    & \\[-1pt]
\hline
 6 &  1 &    &    &    &  1 &    &    & 16 &    &    &    &    &    &    &    &  2 &    &    &    & \\[-1pt]
\hline
 6 &    &  1 &    &  1 &    &    &    & 16 &    &    &    &    &    &    &    &    &  2 &    &    & \\[-1pt]
\hline
 5 &    &  1 &    &    &    &    &    & 22 &  2 &    &    &    &    &    &    &    &    &    &    & \\[-1pt]
\hline
 5 &    &    &    &  1 &    &    &    & 20 &    &  2 &    &    &    &    &    &    &    &    &    & \\[-1pt]
\hline
 4 &  2 &    &  1 &    &  1 &    &    &  8 &    &    &    &    &  2 &    &    &  2 &    &    &    & \\[-1pt]
\hline
 4 &  1 &  2 &    &    &  1 &    &    &  8 &    &    &    &    &    &  2 &    &  2 &    &    &    & \\[-1pt]
\hline
 4 &    &  3 &    &  1 &    &    &    &  8 &  3 &  2 &    &    &    &    &    &    &    &    &    & 1\\[-1pt]
\hline
 4 &    &  2 &    &  2 &    &    &    &  6 &  4 &  4 &    &    &    &    &    &    &    &    &    & \\[-1pt]
\hline
 4 &    &  3 &    &    &    &    &    & 14 &  2 &    &    &    &    &  2 &    &    &    &    &    & \\[-1pt]
\hline
 4 &    &    &  2 &    &    &    &    & 16 &    &    &    &    &    &    &    &    &    &  2 &    & \\[-1pt]
\hline
 2 &  2 &  1 &    &  2 &    &    &    &  4 &    &  2 &  2 &    &    &    &    &    &  2 &    &    & \\[-1pt]
\hline
 2 &  2 &  1 &    &  1 &    &    &    & 10 &  2 &  2 &  2 &    &    &    &    &    &    &    &    & \\[-1pt]
\hline
 2 &  2 &  1 &    &  1 &    &    &    & 10 &    &  2 &    &    &    &    &    &    &    &    &  2 & \\[-1pt]
\hline
 2 &  2 &  1 &    &  1 &    &    &    & 10 &    &    &  2 &    &    &    &    &    &  2 &    &    & \\[-1pt]
\hline
 2 &  2 &  1 &    &    &    &    &    & 16 &  2 &    &  2 &    &    &    &    &    &    &    &    & \\[-1pt]
\hline
 2 &  2 &    &    &  1 &    &    &    & 14 &    &  2 &  2 &    &    &    &    &    &    &    &    & \\[-1pt]
\hline
 2 &    &  3 &    &    &    &    &    & 14 &  3 &    &    &    &    &    &    &    &    &    &    & 1\\[-1pt]
\hline
 2 &    &    &    &    &    &    &    & 26 &    &    &    &    &    &    &    &    &    &    &    & \\[-1pt]
\hline
 1 &  2 &  1 &    &  2 &    &    &    &  4 &    &  4 &    &    &    &    &    &    &    &    &  2 & \\[-1pt]
\hline
 1 &  2 &    &    &  2 &    &    &    &  8 &    &  4 &  2 &    &    &    &    &    &    &    &    & \\[-1pt]
\hline
 1 &  2 &  1 &    &    &    &    &    & 16 &    &    &    &    &    &    &    &    &    &    &  2 & \\[-1pt]
\hline
 1 &  2 &    &    &    &    &    &    & 20 &    &    &  2 &    &    &    &    &    &    &    &    & \\[-1pt]
\end{tabular}%
\caption{Classification of flexible LDTS($13$)s.}
\label{tbl:class}
\end{center}
\end{table}

\section{Recursive constructions}
In this section we present, in the form of theorems, some
recursive constructions for flexible Latin directed triple
systems. Using Theorem~\ref{33}, we express these in terms of
decompositions of the complete graph $K_n$ into triangles and
$k$-gonal bipyramids $O_k$. We represent the latter by the
notation $[N:E_1,E_2,\ldots,E_k:S]$ where $N$ and $S$ are the
poles and $(E_1,E_2,\ldots,E_k)$ is the equator cycle. The first
two constructions are ``doubling'' and ``trebling'' constructions
respectively which often apply for combinatorial designs.
\begin{theorem}
If there exists a flexible LDTS($n$) based on a decomposition of
the complete graph $K_n$ into triangles and graphs $O_k$, $k \geq
4$ and even, then there exists a flexible LDTS($2n+1$).
\end{theorem}
\textbf{Proof.}
Let $(V,\mathcal{B})$ be a decomposition of the complete graph
$K_n$ into triangles and graphs $O_k$, $k \geq 4$ and even, as
stated in the statement of the theorem where
$V=\{0,1,\ldots,n-1\}$. Let $V'=\{\,x':x\in V\,\}$ and $W=V\cup
V'\cup\{\infty\}$. Construct a decomposition of the complete graph
$K_{2n+1}$ on the set $W$ as follows. For all
$[N:E_1,E_2,E_3,\ldots,E_{2l}: S] \in \mathcal{B}$, assign
$[N:E_1,E_2,E_3,\ldots,E_{2l}: S]$,
$[N:E'_1,E'_2,E'_3,\ldots,E'_{2l}: S]$,
$[N':E_1,E'_2,E_3,\ldots,E'_{2l}: S']$,
$[N':E'_1,E_2,E'_3,\ldots,E_{2l}: S'] \in \mathcal{B'}$. For all
$\{x,y,z\} \in \mathcal{B}$, assign $[x:y,z,y',z':x'] \in
\mathcal{B'}$. Further let $\{x,x',\infty\} \in \mathcal{B'}$ for
all $x \in V$. Then $(W,\mathcal{B'})$ is a decomposition of the
complete graph $K_{2n+1}$ into triangles and $k$-gonal bipyramids.
\eproof
\begin{theorem}
If there exists a flexible LDTS($n$) based on a decomposition of
the complete graph $K_n$ into triangles and graphs $O_k$, $k \geq
4$ and even, then there exists a flexible LDTS($3n$).
\end{theorem}
\textbf{Proof.}
Let $(V,\mathcal{B})$ be a decomposition of the complete graph
$K_n$ into triangles and graphs $O_k$, $k \geq 4$ and even, as
stated in the statement of the theorem where
$V=\{0,1,\ldots,n-1\}$. Let $V'=\{\,x':x\in V\,\}$,
$V''=\{\,x'':x\in V\,\}$ and $W=V \cup V' \cup V''$. Construct a
decomposition of the complete graph $K_{3n}$ on the set $W$ as
follows. For all $[N:E_1,E_2,E_3,\ldots,E_{2l}: S] \in
\mathcal{B}$, assign $[N:E_1,E_2,E_3,\ldots,E_{2l}:S]$,
$[N:E'_1,E''_2,E'_3,\ldots,E''_{2l}:S]$,
$[N:E''_1,E'_2,E''_3,\ldots,E'_{2l}:S]$,
$[N':E'_1,E'_2,E'_3,\ldots,E'_{2l}:S']$,
$[N':E_1,E''_2,E_3,\ldots,E''_{2l}:S']$,
$[N':E''_1,E_2,E''_3,\ldots,E_{2l}:S']$,
$[N'':E''_1,E''_2,E''_3,$ $\ldots,E''_{2l}:S'']$,
$[N'':E_1,E'_2,E_3,\ldots,E'_{2l}:S'']$,
$[N'':E'_1,E_2,E'_3,\ldots,E_{2l}:S''] \in \mathcal{B'}$. For all
$\{x,y,z\} \in \mathcal{B}$, assign $\{x,y,z'\}$, $\{x,y',z''\}$,
$\{x,y'',z\}$, $[x':y,z,y',z',y'',z'':x''] \in \mathcal{B'}$.
Further let $\{x,x',x''\} \in \mathcal{B'}$ for all $x \in V$.
Note that for all $\{x,y,z\} \in \mathcal{B}$, the construction
yields a flexible LDTS(9) on the point set
$\{x,x',x'',y,y',y'',z,z',z''\}$. Then $(W,\mathcal{B'})$ is a
decomposition of the complete graph $K_{3n}$ into triangles and
$k$-gonal bipyramids.
\eproof

~\\
\indent The next construction is also a ``doubling'' construction and
employs a Hamiltonian decomposition.

\begin{theorem}
If there exists a flexible LDTS($n$), then there exists a flexible
LDTS($2n+1$).
\end{theorem}
\textbf{Proof.}
Let $(V,\mathcal{B})$ be a flexible LDTS($n$) where
$V=\{x_1,x_2,\ldots,x_n\}$ disjoint from the set
$\mathcal{Z}_{n+1}$ and $K_{n+1}$ be the complete graph on
$\mathcal{Z}_{n+1}$.

Suppose that $n$ is even. Take a decomposition of $K_{n+1}$ into
$n/2$ disjoint Hamiltonian cycles $H_i$, $1 \leq i \leq n/2$. For
each $i$, construct a $(n+1)$-gonal bipyramid
$[x_{2i-1}:H_i:x_{2i}]$ and let $\mathcal{B'}$ be the set of
unidirectional triples obtained from these bipyramids. Then $(V
\cup \mathcal{Z}_{n+1}, \mathcal{B} \cup \mathcal{B'})$ is a
flexible LDTS($2n+1$).

Now suppose that $n$ is odd. Remove a one-factor $F$ from
$K_{n+1}$ and proceed as in the even case using a decomposition of
the graph $K_{n+1} \setminus F$ into $(n-1)/2$ disjoint
Hamiltonian cycles $H_i$, $1 \leq i \leq (n-1)/2$. Further let
$\mathcal{T}$ be the set of Steiner triples $\{a,b,x_n\}$ where
the edge $\{a,b\} \in F$. Then $(V \cup \mathcal{Z}_{n+1},
\mathcal{B} \cup \mathcal{B'} \cup \mathcal{T})$ is a flexible
LDTS($2n+1$).
\eproof

~\\
\indent We remark that in the proof of the theorem we may replace the
Hamiltonian decomposition by any 2-factorization of the relevant
graph. In this respect, a particularly elegant and easy way of
implementing the construction is for each $d \in \mathcal{Z}_{n+1}
\setminus \{0\}$ and $i \in \mathcal {Z}_{n+1}$, assign $\langle
i, x_d, i+d \rangle \in \mathcal{B'}$.

A directed triple system, $(V,\mathcal{B})$, is said to be
\emph{pure} if $\langle x, y, z \rangle \in \mathcal{B}
\Rightarrow \langle z, y, x \rangle \notin \mathcal{B}$. In the
construction described in the above theorem, if $n$ is even and
the LDTS($n$) is pure then so is the LDTS($2n+1$). The
DTS-quasigroups obtained from pure Latin directed triple systems
are anti-commutative.

The final construction is in a similar vein to the previous
construction. We need some further definitions. In a Steiner
triple system, STS($n$), a \emph{parallel class} is a set of
blocks which collectively contain every point of the STS($n$)
precisely once. A \emph{Kirkman triple system} of order~$n$,
KTS($n$), is a triple $(V,\mathcal{B},\mathcal{R})$ where
$(V,\mathcal{B})$ is an STS($n$) and $\mathcal{R}$ is a
\emph{partition} or \emph{resolution} of the set of blocks
$\mathcal{B}$ into parallel classes. Such systems exist if and
only if $n \equiv$ 3 (mod 6), \cite {Lu}, \cite{RCW}.
\begin{theorem}
If there exists a flexible LDTS($2n$), then there exists a
flexible LDTS($6s+3+2n$) for all $s \geq (n-1)/3$.
\end{theorem}
\textbf{Proof.}
Let $(V,\mathcal{B})$ be a flexible LDTS($2n$) where
$V=\{1,2,\ldots,2n\}$ and $(W,\mathcal{S},\mathcal{R})$ be a
KTS($6s+3$) where the set $W$ is disjoint from the set $V$. The
partition $\mathcal{R}$ consists of $3s+1$ parallel classes
$\Pi_i$, $1 \leq i \leq 3s+1$. For the first $n$ parallel classes
$\Pi_i$, $1 \leq i \leq n$, construct trigonal bipyramids
$[2i-1:x,y,z:2i]$ where $\{x,y,z\} \in \Pi_i$, and then decompose
these into unidirectional triples
\[ \dtr{x,2i-1,y}, \dtr{y,2i-1,z}, \dtr{z,2i-1,x}, \dtr{y,2i,x}, \dtr{z,2i,y}, \dtr{x,2i,z}. \]
Denote this set of unidirectional triples by $\mathcal{B'}$. The
remaining parallel classes together form the set of unordered or
Steiner triples  $\mathcal{T} = \bigcup_{i=n+1}^{3s+1}\Pi_i$. Then
$(V \cup W, \mathcal{B} \cup \mathcal{B'} \cup \mathcal{T})$ is a
flexible LDTS($6s+3+2n$).
\eproof

\section{Existence of flexible Latin directed triple systems}
In this section we determine the existence spectrum of flexible
LDTS($n$). For $n$ odd, this has previously been done
in~\cite{ldts} but the proof is short, so we include it for
completeness. We will need a definition. In a Steiner triple
system, a collection of four triples on six points is called a
\emph{Pasch configuration}. It is easily seen that this structure
necessarily has the form $\str{a,b,c}$, $\str{a,y,z}$,
$\str{x,b,z}$, $\str{x,y,c}$. Given such a Pasch configuration, we
will replace it by transitive triples $\dtr{a,b,c}$,
$\dtr{a,y,z}$, $\dtr{x,b,z}$, $\dtr{x,y,c}$, $\dtr{z,y,x}$,
$\dtr{c,b,x}$, $\dtr{c,y,a}$, $\dtr{z,b,a}$. These can be thought
of as a partial LDTS($6$) but a crucial point is that they satisfy
the flexible law. We will denote this collection of eight
transitive triples by~$\mathcal{P}$. Part of the proof also uses a
standard technique, known as Wilson's fundamental construction,
for which we need the concept of a \textit{group divisible design}
(GDD). A 3-GDD of type $g^u$ is an ordered triple
($V,\mathcal{G},\mathcal{B}$) where $V$ is a base set of
cardinality $v=gu$, $\mathcal{G}$ is a partition of $V$ into $u$
subsets of cardinality $g$ called \textit{groups} and
$\mathcal{B}$ is a family of triples called \textit{blocks} which
collectively have the property that every pair of elements from
\underline{different} groups occur in precisely one block but no
pair of elements from the \underline{same} group occur at all. In
the proof for $n$ even, we will also need 3-GDDs of type $g^u
m^1$. These are defined analogously, with the base set $V$ being
of cardinality $v=gu+m$ and the partition $G$ being into $u$
subsets of cardinality $g$ and one set of cardinality $m$.
Necessary and sufficient conditions for 3-GDDs of type $g^u$ were
determined in~\cite{H} and for 3-GDDs of type $g^um^1$
in~\cite{CHR}; a convenient reference is~\cite{Ge} where the
existence of all the GDDs that are used can be verified.

\begin{proposition}\label{propflex13mod6}
There exists a proper flexible LDTS($n$) for all $n\equiv 1,3\pmod{6}$.
\end{proposition}
\textbf{Proof.}
(a) $n\equiv 3,7\pmod{12}$. Put $m=(n-1)/2$ and choose an
STS($m$), $(V,\mathcal{B})$. Let $V'=\{\,x':x \in V\,\}$ and $W=V
\cup V' \cup\{\infty\}$. Construct a collection of triples
$\mathcal{B'}$ as follows. For all $\str{x,y,z} \in \mathcal{B}$,
assign $\str{x,y,z}, \str{x,y',z'}, \str{x',y,z'}$, $\str{x',y',z}
\in \mathcal{B}'$. Further let $\str{x,x',\infty} \in
\mathcal{B}'$ for all $x \in V$. Then $(W,\mathcal{B}')$ is an
STS($n$). In order to obtain a LDTS($n$) replace each Pasch
configuration as above by the set $\mathcal{P}$ of transitive
triples, and retain the sets containing the point~$\infty$ as
Steiner triples. Because the LDTS($n$) is constructed of flexible
components, i.e.\ just the flexible partial LDTS($6$),
$\mathcal{P}$, and the trivial Steiner quasigroup on 3 points, it
is also flexible.

(b) $n\equiv 9\pmod{12}$. Put $m=(n-3)/2$ and choose an
STS($m$), $(V,\mathcal{B})$ which contains a parallel class.
Denote this parallel class by $\Pi$. Let $V'=\{\,x':x\in V\,\}$
and $W=V\cup V'\cup\{\infty_1,\infty_2,\infty_3\}$. Construct a
collection of triples $\mathcal{B}'$ as follows. For all
$\str{x,y,z}\in\Pi$, assign $\str{x,y,z}$, $\str{x',y',z'}$,
$\str{x,x',\infty_1}$, $\str{y,y',\infty_1}$,
$\str{z,z',\infty_1}$, $\str{x,y',\infty_2}$,
$\str{y,z',\infty_2}$, $\str{z,x',\infty_2}$,
$\str{x,z',\infty_3}$, $\str{y,x',\infty_3}$,
$\str{z,y',\infty_3}\in\mathcal{B}'$ and for all
$\str{x,y,z}\in\mathcal{B}\setminus\Pi$, assign $\str{x,y,z}$,
$\str{x,y',z'}$, $\str{x',y,z'}$, $\str{x',y',z}\in\mathcal{B}'$.
Finally let $\str{\infty_1,\infty_2,\infty_3}\in\mathcal{B}'$.
Then $(W,\mathcal{B}')$ is an STS($n$). In order to obtain an
LDTS($n$), replace each Pasch configuration by the
set~$\mathcal{P}$ of transitive triples in the same way as in~(a).
Further replace each collection of eleven triples corresponding to
each block of the parallel class, together with the set
$\str{\infty_1,\infty_2,\infty_3}$, by the flexible LDTS($9$) from
Example~\ref{flex9}, ensuring that the triple
$\str{\infty_1,\infty_2,\infty_3}$ corresponds to a Steiner triple
for each collection.

(c) $n\equiv 1\pmod{12}$.
Take a 3-GDD of type $6^s, s \ge 3$. Inflate each point by
a factor 2 and adjoin an extra point $\infty$. On each inflated group, together with
the point $\infty$, place the flexible LDTS(13) given in Example~\ref{flex13}. On each
inflated block place the set $\mathcal{P}$ of transitive triples $\dtr{a, b, c}$, $\dtr{a, y, z}$,
$\dtr{x, b, z}$, $\dtr{x, y, c}$, $\dtr{z, y, x}$, $\dtr{c, b, x}$, $\dtr{c, y, a}$, $\dtr{z, b, a}$,
with the three sets of points $\{a,x\}, \{b,y\},\{c,z\}$ as the inflated points in the
three groups. We will use $\mathcal{P}$ in this manner throughout. This misses the
value $n=25$ but this can also be constructed in a similar manner by taking a 3-GDD of
type $4^3$, inflating each point by a factor~$2$ and adjoining an extra point~$\infty$.
On each inflated group, together with the point~$\infty$, place the flexible LDTS($9$)
from Example~\ref{flex9} and on each inflated block, place the set of transitive triples~$\mathcal{P}$.
\eproof

~\\
\indent The determination of the spectrum of flexible LDTS($n$) for $n$ even is more intricate,
mainly because there exist no LDTS($n$) for $n=4$, $6$, and $10$, and the only two
LDTS($12$)s are not flexible, \cite{ldts}. The smallest even order flexible system is
LDTS($16$). Again we will use Wilson's fundamental construction, but we will need a
variety of 3-GDDs and initial systems. Flexible LDTS($n$) for $n=16$, $18$, $22$, $24$,
$28$, $30$, $34$, $36$, and $40$ are given as Examples~\ref{flex16} to~\ref{flex40} in
the Appendix and were all found by computer search.

We can now prove a series of propositions
\begin{proposition}
There exists a flexible LDTS($n$) for all $n\equiv 0,16\pmod{24}$.
\end{proposition}
\textbf{Proof.}
(a) $n\equiv0\pmod{48}$. Take a 3-GDD of type $8^{3s}$,
$s\geq1$. Inflate each point by a factor $2$. On each inflated
group place a flexible LDTS($16$) and on each inflated block,
place the set of transitive triples~$\mathcal{P}$.

(b) $n\equiv16\pmod{48}$.
Proceed as in (a) starting with a 3-GDD of type $8^{3s+1}$, $s\geq 1$.

(c) $n\equiv24\pmod{48}$.
Again proceed as in (a) starting with a 3-GDD of type $8^{3s}12^1$, $s\geq1$,
and in addition on the inflated group of cardinality~$12$ place a flexible LDTS($24$).

(d) $n\equiv40\pmod{48}$. Proceed as in~(c) starting with a 3-GDD of
type $8^{3s+1}12^1$, $s\geq1$. This misses the value $n=40$, but a flexible LDTS($40$)
is Example~\ref{flex40} in the Appendix.
\eproof

\begin{proposition}
There exists a flexible LDTS($n$) for all $n\equiv4,12\pmod{24}$, $n\geq28$ except $n=52$, $60$, $76$, $84$.
\end{proposition}
\textbf{Proof.}
(a) $n\equiv4\pmod{24}$.
Take a 3-GDD of type $12^s14^1$, $s\geq3$.
Inflate each point by a factor~$2$.
On each inflated group place a flexible LDTS($24$) or LDTS($28$) as appropriate and
on each inflated block, place the set of transitive triples~$\mathcal{P}$.
This misses the values $n=52$ and $76$.

(b) $n\equiv12\pmod{24}$.
Proceed as in (a) starting with a 3-GDD of type $12^s18^1$, $s\geq3$ and using a
flexible LDTS($36$) instead of an LDTS($28$). This misses the values $n=60$ and $84$.
\eproof

~\\
\indent Before dealing with the next two residue classes we will need two further flexible
systems of orders $42$ and $46$. These can be obtained using the following elementary
construction techniques.

\begin{proposition}\label{44}
\begin{enumerate}
\item[(i)] If there exists a flexible LDTS($n$), then there exists a flexible LDTS($3n-2$).

\item[(ii)] If there exists a flexible LDTS($n$) containing a Steiner triple, then
there exists a flexible LDTS($3n-6$), also containing a Steiner triple.
\end{enumerate}
\end{proposition}
\textbf{Proof.}
(i) Take three copies of the LDTS($n$) on point sets $\{\infty,0_i,1_i,\dots,(n-2)_i\}$,
$i\in\{0,1,2\}$ respectively. Then take a Latin square $L(i,j)$ of order~$n-1$ on the set
$\{0,1,\dots,n-2\}$ and adjoin all Steiner triples $\{i_0,j_1,L(i,j)_2\}$,
$0\leq i\leq n-2$, $0\leq j\leq n-2$.

(ii) Take three copies of the LDTS($n$) on point sets $\{\infty_1,\infty_2,\infty_3,0_i,1_i,$\\$\dots,(n-4)_i\}$,
$i\in\{0,1,2\}$ respectively, where $\{\infty_1,\infty_2,\infty_3\}$ is a Steiner triple
in all three systems. Then take a Latin square $L(i,j)$ of order $n-3$ on the set $\{0,1,\dots,n-4\}$
and adjoin all Steiner triples $\{i_0,j_1,L(i,j)_2\}$, $0\leq i\leq n-4$, $0\leq j\leq n-4$.
\eproof

~\\
\indent We now have the required LDTS($42$) and LDTS($46$) using the
flexible LDTS($16$) from Example~\ref{flex16}.

\begin{proposition}
There exists a flexible LDTS($n$) for all $n\equiv6,10\pmod{12}$, $n\geq18$ except
$n=58$, $66$, $70$, $78$, $82$.
\end{proposition}
\textbf{Proof.}
(a) $n\equiv18\pmod{36}$.
Take a 3-GDD of type $9^{2s+1}$, $s\geq1$.
Inflate each point by a factor~$2$.
On each inflated group place a flexible LDTS(18) and on each inflated block, place the
set of transitive triples~$\mathcal{P}$.

(b) $n\equiv 22,30,34,42,46\pmod{36}$.
Proceed as in~(a) starting with a 3-GDD of type $9^{2s}m^1$, $s\geq2$ where $m\in\{11,15,17,21,23\}$
and in addition on the single larger inflated group, place a flexible LDTS($2m$).
\eproof

~\\
\indent It remains to deal with the nine exceptional values.

\begin{proposition}
There exist flexible LDTS($n$) for $n\in\{52, 58, 60, 66, 70,$\\ $76, 78, 82, 84\}$.
\end{proposition}
\textbf{Proof.}
(a) The values $n=52,70,82$ can be obtained from the construction of Proposition~\ref{44}~(i)
using examples on $18$, $24$, $28$ points respectively given in the Appendix.

(b) The values $n=60,66,78,84$ can be obtained from the construction of Proposition~\ref{44}~(ii)
using examples on $22$, $24$, $28$, $30$ points respectively again given in the Appendix.

(c) For $n=76$, take a 3-GDD of type $15^5$ and on each group together with a
further point~$\infty$, place the flexible LDTS($16$) in Example~\ref{flex16}. Each block
is a Steiner triple.

(d) The value $n=58$ is the most difficult. We shall use the same approach
as for the non-flexible case. Define sets
$\mathcal{N}=\{\,\infty_j:0 \leq j \leq 6\,\}$,
$\mathcal{M}_k=\{\,i_k:0 \leq i \leq 16\,\}$, $k=0$, $1$, $2$.
Take three copies of the flexible LDTS($24$) containing an LDTS($7$) as a subsystem,
constructed as in Example~\ref{flex24} on point sets $\mathcal{N} \cup \mathcal{M}_0$,
$\mathcal{N} \cup \mathcal{M}_1$, $\mathcal{N} \cup \mathcal{M}_2$ respectively, in each
case with the LDTS($7$) on the set $\mathcal{N}$. Then take a Latin square $L(i,j)$ of
side~17 on the set $\{0,1,\dots,16\}$ and adjoin all Steiner triples $\str{i_0,j_1,L(i,j)_2}$.
\eproof

~\\
\indent Collecting together all the results in this section gives the following theorem.

\begin{theorem}
The existence spectrum of flexible LDTS($n$)s is $n\equiv0,1\pmod{3}$, $n\neq4$, $6$, $10$, $12$.
\end{theorem}

\appendix
\section*{Appendix. Examples of flexible LDTSs}
\renewcommand\thesection{A}
\setcounter{theorem}{0}
The following examples were obtained by
computer with the help of the model builder Mace4 \cite{Mace4}
using an algebraic description of a DTS-quasigroup, see
\cite{basics}. We denote the elements $(i,j)\in\mathcal{Z}_{m}
\times \mathcal{Z}_{n}$ as~$i_j$. For simplicity, we omit commas
from the triples.

{\sloppy
\begin{example}\label{flex16}
Flexible LDTS(16).\\
$V = (\mathcal{Z}_3 \times \mathcal{Z}_5) \cup \{\infty\}$.\\
The triples are obtained from the following starter blocks under the action of the mapping $i_j \mapsto (i + 1)_j$,
with $\infty$ as a fixed point.\\
The starter blocks for $\mathcal{T}$ are
$\str{0_0\,1_0\,2_0}$, $\str{0_0\,1_3\,1_4}$, $\str{0_1\,2_2\,0_4}$, $\str{0_1\,0_3\,2_3}$, $\str{0_2\,2_3\,\infty}$,
and for $\mathcal{D}$ are
$\dtr{1_2\,0_0\,0_4}$, $\dtr{0_4\,0_0\,2_3}$, $\dtr{2_3\,0_0\,1_2}$,
$\dtr{1_2\,1_4\,2_3}$, $\dtr{2_3\,1_4\,0_4}$, $\dtr{0_4\,1_4\,1_2}$,
$\dtr{0_0\,0_1\,1_1}$, $\dtr{1_1\,0_1\,2_4}$, $\dtr{2_4\,0_1\,0_0}$,
$\dtr{0_0\,\infty\,2_4}$, $\dtr{2_4\,\infty\,1_1}$, $\dtr{1_1\,\infty\,0_0}$,
$\dtr{1_0\,0_2\,0_1}$, $\dtr{0_1\,0_2\,1_2}$, $\dtr{1_2\,0_2\,1_0}$,
$\dtr{1_0\,1_3\,1_2}$, $\dtr{1_2\,1_3\,0_1}$, $\dtr{0_1\,1_3\,1_0}$.
\end{example}

\begin{example}\label{flex18}
Flexible LDTS(18).\\
$V = \mathcal{Z}_3 \times \mathcal{Z}_6$.\\
The triples are obtained from the following starter blocks under the action of the mapping $i_j \mapsto (i + 1)_j$.\\
The starter blocks for $\mathcal{T}$ are
$\str{0_0\,1_0\,2_0}$, $\str{0_1\,0_4\,1_5}$, $\str{0_2\,1_2\,2_2}$, $\str{0_4\,1_4\,2_4}$,
and for $\mathcal{D}$ are
$\dtr{1_3\,0_0\,0_5}$, $\dtr{0_5\,0_0\,0_4}$, $\dtr{0_4\,0_0\,1_3}$,
$\dtr{1_3\,2_3\,0_4}$, $\dtr{0_4\,2_3\,0_5}$, $\dtr{0_5\,2_3\,1_3}$,
$\dtr{2_0\,0_1\,2_1}$, $\dtr{2_1\,0_1\,1_3}$, $\dtr{1_3\,0_1\,2_0}$,
$\dtr{2_0\,1_4\,1_3}$, $\dtr{1_3\,1_4\,2_1}$, $\dtr{2_1\,1_4\,2_0}$,
$\dtr{1_0\,0_2\,0_1}$, $\dtr{0_1\,0_2\,2_5}$, $\dtr{2_5\,0_2\,1_0}$,
$\dtr{1_0\,0_5\,2_5}$, $\dtr{2_5\,0_5\,0_1}$, $\dtr{0_1\,0_5\,1_0}$,
$\dtr{0_0\,0_3\,0_2}$, $\dtr{0_2\,0_3\,0_5}$, $\dtr{0_5\,0_3\,2_2}$,
$\dtr{2_2\,0_3\,0_1}$, $\dtr{0_1\,0_3\,1_2}$, $\dtr{1_2\,0_3\,0_0}$,
$\dtr{0_0\,1_4\,1_2}$, $\dtr{1_2\,1_4\,0_1}$, $\dtr{0_1\,1_4\,2_2}$,
$\dtr{2_2\,1_4\,0_5}$, $\dtr{0_5\,1_4\,0_2}$, $\dtr{0_2\,1_4\,0_0}$.
\end{example}

\begin{example}\label{flex22}
Flexible LDTS(22).\\
$V = \mathcal{Z}_{11} \times \mathcal{Z}_2$.\\
The triples are obtained from the following starter blocks under the action of the mapping $i_j \mapsto (i + 1)_j$.\\
The starter blocks for $\mathcal{T}$ are
$\str{0_0\,1_0\,3_0}$, $\str{0_0\,5_1\,10_1}$,
and for $\mathcal{D}$ are
$\dtr{4_0\,0_0\,1_1}$, $\dtr{1_1\,0_0\,6_0}$, $\dtr{6_0\,0_0\,9_1}$,
$\dtr{9_1\,0_0\,0_1}$, $\dtr{0_1\,0_0\,4_0}$, $\dtr{4_0\,8_1\,0_1}$,
$\dtr{0_1\,8_1\,9_1}$, $\dtr{9_1\,8_1\,6_0}$, $\dtr{6_0\,8_1\,1_1}$,
$\dtr{1_1\,8_1\,4_0}$.
\end{example}

\begin{example}\label{flex24}
Flexible LDTS(24) containing an LDTS(7) as a subsystem.\\
$V=\mathcal{Z}_{8} \times \mathcal{Z}_{3}$.\\
This system in fact contains three disjoint LDTS(7)s on point sets
$\{0_i,1_i,2_i,$\\$3_i,4_i,5_i,6_i\}$, $i \in \{0,1,2\}$,
respectively.
The triples are obtained from the following starter blocks under the action of the mapping $i_j \mapsto i_{j+1}$.\\
The starter blocks for the Steiner triples $\mathcal{T}_1$ are
$\str{0_0\,1_0\,2_0}$, $\str{0_0\,3_0\,4_0}$,\\ $\str{0_0\,5_0\,6_0}$,
and for the unidirectional triples $\mathcal{D}_1$ are
$\dtr{3_0\,1_0\,5_0}$, $\dtr{5_0\,1_0\,4_0}$, $\dtr{4_0\,1_0\,6_0}$, $\dtr{6_0\,1_0\,3_0}$
$\dtr{3_0\,2_0\,6_0}$, $\dtr{6_0\,2_0\,4_0}$, $\dtr{4_0\,2_0\,5_0}$, $\dtr{5_0\,2_0\,3_0}$.
The starter blocks for the remaining Steiner triples $\mathcal{T}_2$ are
$\str{0_0\,6_1\,2_2}$,\\ $\str{1_0\,1_1\,1_2}$, $\str{1_0\,5_1\,5_2}$, $\str{1_0\,6_1\,6_2}$,
$\str{2_0\,3_1\,4_2}$, $\str{2_0\,5_1\,3_2}$, $\str{2_0\,5_2\,7_1}$, $\str{3_0\,5_1\,4_2}$,
$\str{4_0\,4_1\,4_2}$, $\str{4_0\,5_1\,6_2}$, $\str{5_0\,6_2\,7_1}$,
and for the unidirectional triples $\mathcal{D}_2$ are
$\dtr{1_1\,0_0\,7_0}$, $\dtr{7_0\,0_0\,4_2}$, $\dtr{4_2\,0_0\,1_1}$, $\dtr{1_1\,7_2\,4_2}$,
$\dtr{4_2\,7_2\,7_0}$, $\dtr{7_0\,7_2\,1_1}$, $\dtr{3_1\,1_0\,7_0}$, $\dtr{7_0\,1_0\,3_2}$,
$\dtr{3_2\,1_0\,3_1}$, $\dtr{3_1\,6_0\,3_2}$, $\dtr{3_2\,6_0\,7_0}$, $\dtr{7_0\,6_0\,3_1}$,
$\dtr{0_1\,3_0\,7_0}$, $\dtr{7_0\,3_0\,0_2}$, $\dtr{0_2\,3_0\,0_1}$, $\dtr{0_1\,5_0\,0_2}$,
$\dtr{0_2\,5_0\,7_0}$, $\dtr{7_0\,5_0\,0_1}$, $\dtr{1_1\,2_0\,0_2}$, $\dtr{0_2\,2_0\,6_1}$,
$\dtr{6_1\,2_0\,7_2}$, $\dtr{7_2\,2_0\,2_2}$, $\dtr{2_2\,2_0\,1_1}$, $\dtr{1_1\,4_0\,2_2}$,
$\dtr{2_2\,4_0\,7_2}$, $\dtr{7_2\,4_0\,6_1}$, $\dtr{6_1\,4_0\,0_2}$, $\dtr{0_2\,4_0\,1_1}$.
Put $\mathcal{T}=\mathcal{T}_1\cup\mathcal{T}_2$ and $\mathcal{D}=\mathcal{D}_1\cup\mathcal{D}_2$.
\end{example}

\begin{example}\label{flex28}
Flexible LDTS(28).\\
$V = \mathcal{Z}_{14} \times \mathcal{Z}_2$.\\
The triples are obtained from the following starter blocks under the action of the mapping $i_j \mapsto (i + 1)_j$.\\
The starter block for $\mathcal{T}$ is $\str{0_0\,1_0\,3_0}$, and for $\mathcal{D}$ are
$\dtr{4_0\,0_0\,3_1}$, $\dtr{3_1\,0_0\,9_0}$, $\dtr{9_0\,0_0\,1_1}$,
$\dtr{1_1\,0_0\,0_1}$, $\dtr{0_1\,0_0\,4_0}$, $\dtr{4_0\,11_1\,0_1}$,
$\dtr{0_1\,11_1\,1_1}$, $\dtr{1_1\,11_1\,9_0}$, $\dtr{9_0\,11_1\,3_1}$,
$\dtr{3_1\,11_1\,4_0}$, $\dtr{2_0\,0_1\,10_0}$, $\dtr{10_0\,0_1\,5_1}$,
$\dtr{5_1\,0_1\,12_1}$, $\dtr{12_1\,0_1\,3_0}$, $\dtr{3_0\,0_1\,9_0}$,
$\dtr{9_0\,0_1\,2_0}$.
\end{example}

\begin{example}\label{flex30}
Flexible LDTS(30).\\
$V = \mathcal{Z}_{15} \times \mathcal{Z}_2$.\\
The triples are obtained from the following starter blocks under the action of the mapping $i_j \mapsto (i + 1)_j$.\\
The starter blocks for $\mathcal{T}$ are
$\str{0_0\,1_0\,3_0}$, $\str{0_0\,5_0\,10_0}$, $\str{0_0\,9_1\,13_1}$, $\str{0_1\,5_1\,10_1}$,
and for $\mathcal{D}$ are
$\dtr{0_0\,6_0\,5_1}$, $\dtr{5_1\,6_0\,8_1}$, $\dtr{8_1\,6_0\,0_0}$,
$\dtr{0_0\,6_1\,8_1}$, $\dtr{8_1\,6_1\,5_1}$, $\dtr{5_1\,6_1\,0_0}$,
$\dtr{9_0\,2_0\,5_1}$, $\dtr{5_1\,2_0\,12_1}$, $\dtr{12_1\,2_0\,6_1}$,
$\dtr{6_1\,2_0\,9_0}$, $\dtr{9_0\,5_0\,6_1}$, $\dtr{6_1\,5_0\,12_1}$,
$\dtr{12_1\,5_0\,5_1}$, $\dtr{5_1\,5_0\,9_0}$.
\end{example}

\begin{example}\label{flex34}
Flexible LDTS(34).\\
$V = \mathcal{Z}_{17} \times \mathcal{Z}_2$.\\
The triples are obtained from the following starter blocks under the action of the mapping $i_j \mapsto (i + 1)_j$.\\
The starter blocks for $\mathcal{T}$ are
$\str{0_0\,1_0\,3_0}$, $\str{0_0\,4_0\,9_0}$, $\str{0_0\,6_0\,0_1}$,
and for $\mathcal{D}$ are
$\dtr{7_0\,0_0\,6_1}$, $\dtr{6_1\,0_0\,2_1}$, $\dtr{2_1\,0_0\,7_0}$,
$\dtr{7_0\,5_1\,2_1}$, $\dtr{2_1\,5_1\,6_1}$, $\dtr{6_1\,5_1\,7_0}$,
$\dtr{0_1\,9_0\,6_1}$, $\dtr{6_1\,9_0\,16_1}$, $\dtr{16_1\,9_0\,14_1}$,
$\dtr{14_1\,9_0\,5_1}$, $\dtr{5_1\,9_0\,0_1}$, $\dtr{0_1\,13_0\,5_1}$,
$\dtr{5_1\,13_0\,14_1}$,\\ $\dtr{14_1\,13_0\,16_1}$, $\dtr{16_1\,13_0\,6_1}$,
$\dtr{6_1\,13_0\,0_1}$.
\end{example}

\begin{example}\label{flex36}
Flexible LDTS(36).\\
$V = \mathcal{Z}_{18} \times \mathcal{Z}_2$.\\
The triples are obtained from the following starter blocks under the action of the mapping $i_j \mapsto (i + 1)_j$.\\
The starter blocks for $\mathcal{T}$ are
$\str{0_0\,1_0\,3_0}$, $\str{0_0\,6_0\,12_0}$, $\str{0_1\,6_1\,12_1}$,
and for $\mathcal{D}$ are
$\dtr{15_0\,4_0\,16_1}$, $\dtr{16_1\,4_0\,17_1}$, $\dtr{17_1\,4_0\,15_0}$,
$\dtr{15_0\,6_1\,17_1}$, $\dtr{17_1\,6_1\,16_1}$,\\ $\dtr{16_1\,6_1\,15_0}$,
$\dtr{5_0\,0_0\,5_1}$, $\dtr{5_1\,0_0\,14_1}$, $\dtr{14_1\,0_0\,14_0}$,
$\dtr{14_0\,0_0\,5_0}$, $\dtr{2_0\,8_1\,6_1}$,\\ $\dtr{6_1\,8_1\,16_0}$,
$\dtr{16_0\,8_1\,13_1}$, $\dtr{13_1\,8_1\,10_0}$, $\dtr{10_0\,8_1\,2_0}$,
$\dtr{2_0\,9_1\,10_0}$, $\dtr{10_0\,9_1\,13_1}$,\\ $\dtr{13_1\,9_1\,16_0}$,
$\dtr{16_0\,9_1\,6_1}$, $\dtr{6_1\,9_1\,2_0}$.
\end{example}

\begin{example}\label{flex40}
Flexible LDTS(40).\\
$V = \mathcal{Z}_{20} \times \mathcal{Z}_2$.\\
The triples are obtained from the following starter blocks under the action of the mapping $i_j \mapsto (i + 1)_j$.\\
The starter blocks for $\mathcal{T}$ are
$\str{0_0\,1_0\,3_0}$, $\str{0_0\,4_0\,9_0}$, $\str{0_0\,8_0\,0_1}$,
and for $\mathcal{D}$ are
$\dtr{0_0\,5_1\,9_1}$, $\dtr{9_1\,5_1\,6_0}$, $\dtr{6_0\,5_1\,0_0}$,
$\dtr{0_0\,14_1\,6_0}$, $\dtr{6_0\,14_1\,9_1}$, $\dtr{9_1\,14_1\,0_0}$,
$\dtr{13_0\,0_0\,15_1}$, $\dtr{15_1\,0_0\,17_1}$, $\dtr{17_1\,0_0\,13_0}$,
$\dtr{13_0\,3_1\,17_1}$, $\dtr{17_1\,3_1\,15_1}$, $\dtr{15_1\,3_1\,13_0}$,
$\dtr{1_0\,7_1\,8_1}$, $\dtr{8_1\,7_1\,18_1}$, $\dtr{18_1\,7_1\,11_0}$,
$\dtr{11_0\,7_1\,4_1}$, $\dtr{4_1\,7_1\,6_0}$, $\dtr{6_0\,7_1\,16_0}$,
$\dtr{16_0\,7_1\,14_1}$, $\dtr{14_1\,7_1\,1_0}$.
\end{example}
}

\end{document}